\documentclass[11pt]{amsart}
\usepackage[a4paper, margin=1.4in]{geometry}

\usepackage[utf8]{inputenc}
\usepackage[T1]{fontenc}
\usepackage{lmodern}

\usepackage{amssymb}
\usepackage{mathrsfs}

\usepackage[dvipsnames]{xcolor}
\usepackage[shortlabels]{enumitem}
\usepackage{needspace}
\usepackage{tikz-cd}
\usepackage{etoolbox}

\usepackage[numbers]{natbib}
\usepackage[colorlinks=true,linkcolor=cyan,citecolor=cyan,urlcolor=cyan]{hyperref}
\usepackage[capitalize]{cleveref}

\newcounter{dummyitem}
\makeatletter
\newcommand\myitem[1][]{\item[#1]\refstepcounter{dummyitem}\def\@currentlabel{#1}}

\patchcmd{\@setaddresses}{\indent}{\noindent}{}{}
\patchcmd{\@setaddresses}{\indent}{\noindent}{}{}
\patchcmd{\@setaddresses}{\indent}{\noindent}{}{}
\patchcmd{\@setaddresses}{\indent}{\noindent}{}{}
\makeatother

\DeclareRobustCommand{\SkipTocEntry}[5]{} 

\setlist[enumerate,1]{label={\upshape(\arabic*)}}

\newcommand{\ACtext}{}
\newenvironment{AC}{\renewcommand{\ACtext}{*}}{}

\newtheorem{lemma}{{\ACtext}Lemma}[section]
\newtheorem{theorem}[lemma]{{\ACtext}Theorem}
\newtheorem{proposition}[lemma]{{\ACtext}Proposition}
\newtheorem{corollary}[lemma]{{\ACtext}Corollary}

\theoremstyle{definition}
\newtheorem{definition}[lemma]{Definition}
\newtheorem{remark}[lemma]{Remark}
\newtheorem{example}[lemma]{Example}

\renewcommand{\O}{\mathsf O}
\renewcommand{\P}{\mathcal P}
\newcommand{\F}{\mathscr F}
\newcommand{\B}{\mathscr B}

\newcommand{\C}{\mathsf C}
\newcommand{\Sat}{\mathsf{S}}
\newcommand{\R}{\mathsf{R}}

\newcommand{\FiltSE}{\mathsf{Filt_{SE}}}

\newcommand{\s}{\mathsf s}

\newcommand{\upset}{\mathord{\uparrow}}
\newcommand{\downset}{\mathord{\downarrow}}
\renewcommand{\diamond}{\lozenge}
\newcommand{\pprec}{\mathrel{{\prec}\mkern-5mu{\prec}}}

\newcommand{\half}{\mathchoice
{\raisebox{.3ex}{\scalebox{.7}{$\displaystyle\frac12$}}}
{\raisebox{.25ex}{\scalebox{.7}{$\textstyle\frac12$}}}
{\raisebox{.2ex}{\scalebox{.7}{$\scriptstyle\frac12$}}}
{\raisebox{.15ex}{\scalebox{.7}{$\scriptscriptstyle\frac12$}}}
}
\newcommand{\D}{{\half}}
\newcommand{\creg}{{3\half}}

\let\se\subseteq

\let\bve\bigvee
\let\bwe\bigwedge
\let\Om\Omega
\let\up\upset

\let\fse\FiltSE
\let\S\Sat

\title{Local compactness does not always imply spatiality}
\author{G.~Bezhanishvili}
\author{S.~D.~Melzer}
\author{R.~Raviprakash}
\author{A.~L.~Suarez}

\address{New Mexico State University, USA}
\email{guram [at] nmsu [dot] edu}
\email{smelzer [at] nmsu [dot] edu}
\email{prakash2 [at] nmsu [dot] edu}

\address{University of the Western Cape, South Africa}
\email{annalaurasuarez993 [at] gmail [dot] com}


\begin{document}

\begin{abstract}
    It is a well-known result in pointfree topology that every locally compact frame is spatial. Whether this result extends to MT-algebras (McKinsey--Tarski algebras)  was an open problem. We resolve it in the negative by constructing a locally compact sober MT-algebra which is not spatial. 
    We also revisit N\"obeling's largely overlooked approach to pointfree topology from the 1950s. We show that his separation axioms are closely related to those in the theory of MT-algebras with the notable exception of Hausdorffness. We prove that N\"obeling's Spatiality Theorem  implies the well-known Isbell Spatiality Theorem. We then generalize N\"obeling's Spatiality Theorem by proving that each locally compact $T_\D$-algebra is spatial. The proof utilizes the fact that every nontrivial $T_\D$-algebra contains a closed atom, which we show is equivalent to the axiom of choice.
\end{abstract}

\maketitle

\tableofcontents

\section{Introduction}

Pointfree topology seeks to describe topological notions  without direct reference to points. The predominant modern approach is through the theory of frames or locales (see, e.g., \cite{PP12}), which generalize lattices of open sets of  topological spaces. This perspective emerged in the mid-20th century 
and became mainstream after Isbell's influential paper  \cite{Isb72}, which is considered the birth of pointfree topology. Soon after, the first monograph of the subject appeared  \cite{Joh82}, to which we refer for further details and history.

An alternative pointfree approach to topology emerged even earlier, in the celebrated paper by McKinsey and Tarski \cite{MT14}, who generalized Kuratowski's
closure operators on powersets
\cite{Kur22} to 
closure operators on boolean algebras. 
Their work initiated a line of research that passed through Dummett and Lemmon \cite{DL59}, Rasiowa and Sikorski \cite{RS63}, and many others, eventually leading to  the Blok--Esakia theorem \cite{Blo76,Esa76}, becoming a central thread in the study of modal logic (see, e.g., \cite{CZ97}). Yet, this approach largely disappeared from pointfree topology, and was only recently reintroduced in \cite{BR23} through the formalism of McKinsey--Tarski algebras or MT-algebras for short.

Between these developments lies another forgotten contribution: N\"obeling’s 1954 book \cite{Noeb54} (see also \cite{Noeb48,Noeb51,Noeb53}),
which proposed a general theory of topological structure on posets equipped with 
closure operators, of which the MT-algebras are a special case.
In particular, N\"obeling developed 
pointfree versions of separation axioms \cite[Sec.~11]{Noeb54} and local compactness \cite[Sec.~12]{Noeb54}, but his work was largely overlooked as pointfree topology became increasingly presented in the language of frames.

One of our aims is to
revisit N\"obeling’s separation axioms from the perspective of MT-algebras, comparing them to the modern formulations 
given in \cite{BR23}. We show that while most separation axioms align, Nöbeling’s version of Hausdorffness (see \cite[p.~79]{Noeb54}) differs considerably as it does not even imply the $T_1$-separation.
However, we show that it is closely related to the notion of Hausdorffness in frames (see, e.g., \cite[Sec.~III.3]{PP21}).
We also revisit N\"obeling’s notion of local compactness (see \cite[p.~105]{Noeb54}), which generalizes the notion of 
each point having a compact neighborhood, and compare it to the MT-version (see \cite[Def.~4.1]{BR25}), which generalizes the notion of 
each neighborhood of a point containing a compact neighborhood of the point. Paralleling what happens in topological spaces, 
the latter implies the former, with the two being equivalent for Hausdorff MT-algebras. 

The main results of this paper concern spatiality. By N\"obeling's Spatiality Theorem \cite[12.5]{Noeb54}, each compact MT-algebra satisfying the $T_1$-separation is spatial. 
We show that this result implies the well-known Isbell Spatiality Theorem \cite[Thm.~2.1]{Isb72} that each compact subfit frame is spatial. We then generalize the former 
by proving that each locally compact MT-algebra satisfying the $T_\D$-separation is spatial. This we do by establishing that each nontrivial compact MT-algebra satisfying the $T_\D$-separation contains a closed atom, a statement we prove is equivalent to the axiom of choice. 
Based on this, one might expect that the result on spatiality of 
locally compact frames (see, e.g., \cite[Sec.~VII.6]{PP12}) would generalize to MT-algebras. 
We show that this is not the case
by constructing a locally compact sober MT-algebra which is not spatial, thereby confirming the suspicion of \cite[Rem.~4.14]{BR25}. As a byproduct, we answer the first two open problems at the end of \cite{BR23} in the negative.

These results reveal that the spatiality of 
locally compact frames is not a general feature of pointfree topology, but rather a consequence of the limited expressive power of the  formalism of frames. 
Indeed, as follows from \cite{BR23,BR+25}, frames correspond to those MT-algebras that satisfy the $T_\D$-separation,  explaining why every locally compact frame is spatial:  because every locally compact MT-algebra satisfying the $T_\D$-separation is spatial. But the latter result no longer holds if we weaken $T_\D$ to $T_0$, and even strengthening $T_0$ to sober is not sufficient to obtain spatiality.  

\section{Preliminaries}

With each topological space $X$ we associate two lattices: 
the lattice of open sets $\Omega(X)$ and the powerset lattice
$\P(X)$ equipped with
topological interior.
The first leads to the notion of frame and the second to that of MT-algebra.

\needspace{2em}
\begin{definition}
\leavevmode
\begin{enumerate}
    \item A \emph{frame} is a complete lattice in which finite meets distribute over arbitrary joins.
    \item An \emph{MT-algebra} is a complete boolean algebra $M$ equipped with a unary function $\square \colon M \to M$  satisfying Kuratowski's axioms for interior:
    \begin{itemize}
        \item $\square 1 = 1$;
        \item $\square (a \wedge b) = \square a \wedge \square b$;
        \item $\square a \leq a$;
        \item $\square a  \leq \square \square a$.
    \end{itemize}
\end{enumerate}
\end{definition}

We call a frame \emph{spatial} if it is isomorphic to $\Omega(X)$ for some topological space $X$, and an MT-algebra \emph{spatial} if it is isomorphic to $\P(X)$ equipped with its interior operator. Then, a frame is spatial iff $a \nleq b$ implies that there is a completely prime filter containing $a$ but missing $b$ (see, e.g., \cite[Prop.~I.5.1]{PP12}), and an MT-algebra is spatial iff it is atomic (see, e.g., \cite[Thm.~3.22]{BR23}).

Frames and MT-algebras
are closely related.
For any MT-algebra $M$, the image $\O(M)$ of the interior operator $\square$ forms a frame, whose elements are called \emph{open} elements of $M$. 
If $M$ is spatial then so is $\O(M)$ (see \cite[Prop.~4.11]{BR23}), but the converse is not true in general.\label{sentence spatial}
 Going the other way, with each frame $L$ we can associate an MT-algebra $\F L$, called the {\em Funayama envelope} of $L$. It is constructed as 
 the MacNeille completion of the boolean envelope of $L$, where the interior operator $\square$ on $\F L$ 
 is given by the right adjoint of the embedding $L \to \F L$ and $\O(\F L) \cong L$ (see \cite[Sec.~4]{BR23}).\footnote{Alternatively, $\F L$ can be constructed as the booleanization of the frame of nuclei of $L$ (see \cite[Thm.~3.1]{BGJ13}).} 
Thus, every frame arises as the collection of open elements $\O(M)$ of some MT-algebra $M$.

To characterize those MT-algebras that arise as $\F L$ for some frame $L$, we 
recall how the separation axioms in topology generalize to MT-algebras. This will also play a crucial role in the next section, when we compare Nöbeling’s formulations with the ones mentioned here. 
We begin by introducing some terminology.
Let $M$ be an MT-algebra. 
We denote by $\diamond := \neg \square \neg $ the closure operator on $M$, and by $\C(M)$ the image of $\diamond$. The elements of $\C(M)$ are called \emph{closed} elements of $M$. Following the standard topological terminology, an element $a\in M$ is {\em saturated} if it is a meet of open elements. We denote the set of saturated elements by $\S(M)$.

\begin{definition}
    Let $M$ be an MT-algebra and $a \in M$. Then
    \begin{enumerate}
        \item $a$ is a \emph{$T_0$-element} if $a = s \wedge c$ for some $s \in \S(M)$ and $c \in \C(M)$;
        \item $a$ is a \emph{$T_\D$-element} if $a = u \wedge c$ for some $u \in \O(M)$ and $c \in \C(M)$;
        \item $a$ is a \emph{$T_1$-element} if $a$ is closed;
        \item $a$ is a \emph{$T_2$-element} if $a = \bigwedge\{\diamond u \mid a \leq u \in \O(M)\}$.
    \end{enumerate}
\end{definition}

We recall that a subset $S$ of a complete lattice $L$
is \emph{join-dense} if every $a \in L$ is a join
from $S$, and it is  \emph{meet-dense} if every $a \in L$ is a meet from $S$. 

\begin{definition}[Lower separation axioms]
    Let $M$ be an MT-algebra. For $i = 0,\D,1,2$, we say that $M$ is a {\em $T_i$-algebra} provided the $T_i$-elements form a join-dense subset of $M$.
\end{definition}

\begin{definition}\label{completelyrelation}
Let $a, b$ be elements of an MT-algebra $M$. We write $a \prec b$ if $\diamond a \leq \square b$, and $a \pprec b$ if there is a family $\{u_p\} \subseteq \O(M)$, indexed by $p \in [0,1] \cap \mathbb{Q}$, such that
\[
a \leq u_0 \leq u_p \prec u_q \leq u_1 \leq b \quad \text{for all } p < q.
\]
\end{definition}

\begin{remark}\label{dyadic rationals}
Recall that a \emph{dyadic rational} is a rational of the form $\frac{m}{2^n}$ where $m \in \mathbb{Z}$ and $n \in \mathbb{N}$.  The dyadic rationals in $[0,1]$ are those with $0 \leq m \leq 2^n$. Since these form a countable dense subset of $[0,1]$, the family $\{u_p\}$ in \cref{completelyrelation} may be taken to range over dyadic rationals in $[0,1]$.
\end{remark}

Two elements $a, b$ of a boolean algebra are said to be \emph{disjoint} if $a \wedge b = 0$.

\begin{definition}[Higher separation axioms]
    Let $M$ be a $T_1$-algebra. Then
    \begin{enumerate}
        \item $M$ is a {\em $T_3$-algebra} if $u = \bigvee\{v \in \O(M) \mid v \prec u\}$ for all $u \in \O(M)$;
        \item $M$ is a {\em $T_\creg$-algebra} if $u = \bigvee\{v \in \O(M) \mid v \pprec u\}$ for all $u \in \O(M)$;
        \item $M$ is a {\em $T_4$-algebra} if for all disjoint $c,d \in \C(M)$ there exist disjoint $u,v \in \O(M)$ such that $c \leq u$ and $d \leq v$.
    \end{enumerate}
\end{definition}

\begin{remark} \label{Ti-space iff Ti-algebra}
The separation axioms for MT-algebras faithfully extend the ones for topological spaces, i.e., a topological space $X$ is a $T_i$-space iff
the MT-algebra $\P(X)$ is a $T_i$-algebra (see \cite{BR23}).
\end{remark}
The following result characterizes exactly which MT-algebras arise as the Funayama envelope of a frame:

\begin{theorem}[{\cite[Thm.~6.5]{BR23}}] 
\leavevmode
    For an MT-algebra $M$, $M \cong \F \O(M)$ iff 
    $M$ is a $T_\D$-algebra.
    \label{thm: FL Td}
\end{theorem}

This result has the following corollary for 
$T_1$-algebras. 
Recall (see, e.g., \cite[p.~79]{PP12}) that a frame $L$ is \emph{subfit} if for all $a,b \in L$, from $a \nleq b$ it follows that there is $c \in L$ such that $a \vee c = 1$ and $b \vee c \neq 1$.

\begin{theorem}[{\cite[Thm.~6.17]{BR23}}] \label{thm: T1 subfit}
    Let $L$ be a frame. Then
        $\F L$ is $T_1$ iff $L$ is subfit.
\end{theorem}

This will be used 
in \cref{sec:4} to derive Isbell's Spatiality Theorem from Nöbeling's Spatiality Theorem.

\section{N\"obeling's separation axioms}
In this section, we compare N\"obeling's separation axioms to those of the previous section. N\"obeling only dealt with the separation axioms between $T_1$ and $T_4$ (including $T_\creg$).\footnote{N\"obeling also had $T_5$, but we do not discuss it here.} As we will see, his approach is largely equivalent to that of the previous section, with the Hausdorff separation axiom being a notable exception in that the N\"obeling version of it is strictly weaker. On the other hand, it is N\"obeling's version that is closely related to Hausdorffness in frames.  
\begin{definition}
    For an MT-algebra $M$, the N\"obeling separation axioms are defined as follows:
    \begin{enumerate}[label=\textup{(NT$_{\arabic*}$)}, leftmargin=!, labelwidth=3em]
        \item If $a,b \in M$ with $a \nleq b$, then there is $u\in \O(M)$ with $a \nleq u$ and $b \leq u$. \label{NT1}
        
        \item If $a,b \in M$ are nonzero and disjoint, then there are disjoint $u,v \in \O(M)$ such that both $a \wedge u$ and $b \wedge v$ are nonzero. \label{NT2} 
        
        \item If $a \in M$ is nonzero and $c \in \C(M)$ is disjoint from $a$, then there are disjoint $u,v \in \O(M)$ with $a \wedge u$ nonzero and $c \leq v$. \label{NT3} 
        \myitem[{\upshape(NT\ensuremath{_\creg})}] If $a \in M$ is nonzero and $c \in \C(M)$ is disjoint from $a$, then there is a family $\{u_p\} \subseteq \O(M)$,  indexed by dyadic rationals in $[0,1]$, such that 
        $u_p \prec u_q$ for $p < q$, $a \wedge u_0$ is nonzero, and $c$ is disjoint from $u_1$.\label{NT31/2}
        
        \item If $c,d \in \C(M)$ are disjoint, then there are disjoint $u,v \in \O(M)$ such that $c \leq u$ and $d \leq v$.\label{NT4}
    \end{enumerate}
\end{definition}

Observe that $M$ satisfies \ref{NT1} iff $\O(M)$ is meet-dense in $M$, which is equivalent to $\C(M)$ being join-dense in $M$. We thus obtain:

\begin{proposition}
    An MT-algebra $M$ is a $T_1$-algebra iff $M$ satisfies \ref{NT1}.
\end{proposition}

We next show the same for $i = 3,\creg,4$.

\begin{proposition}
    Let $M$ be a $T_1$-algebra. 
    \begin{enumerate}
        \item $M$ is a $T_3$-algebra iff $M$ satisfies \ref{NT3}.
        \item $M$ is a $T_\creg$-algebra iff $M$ satisfies \textup{\ref{NT31/2}}.
        \item $M$ is a $T_4$-algebra iff $M$ satisfies \ref{NT4}.
    \end{enumerate}
\end{proposition}
\begin{proof}
    (1) By \cite[11.9]{Noeb54}, $M$ satisfies \ref{NT3} iff $u = \bigvee\{v \in \O(M) \mid v \prec u\}$ for all $u \in \O(M)$, yielding the result.

    (2) It is sufficient to show that $M$ satisfies \ref{NT31/2} iff $u = \bigvee\{v \in \O(M) \mid v \pprec u\}$ for all $u \in \O(M)$. First suppose $M$ satisfies \ref{NT31/2}. Let $u\in \O(M)$ and set $u'=\bigvee \{v \in \O(M) \mid v \pprec u \}$. If $u \neq u'$, then $a := u \wedge \neg u'$ is nonzero. 
    Since $\neg u$ is disjoint from $a$ and $\neg u \in \C(M)$, by \ref{NT31/2} there is a family $\{u_p\} \subseteq \O(M)$, indexed by dyadic rationals in $[0,1]$, such that $u_p \prec u_q$ for $p <q$, $a \wedge u_0$ is nonzero, and $\neg u$ is disjoint from $u_1$. Therefore, for $p < q$ we have $u_0 \leq u_p \prec u_q \leq u_1 \leq u$. Thus, $u_0 \pprec u$ by \cref{dyadic rationals}. Hence, $u_0 \leq u'$, which implies that $u_0 \wedge a \leq u' \wedge a=0$, a contradiction. Consequently, $u=u'$. 
    
    For the reverse implication, 
    suppose $a \in M$ is nonzero and $c \in \C(M)$ is disjoint from $a$.
    Then $a \leq \neg c$ and $\neg c \in \O(M)$, so
    $\neg c = \bigvee \{v \in \O(M) \mid {v \pprec \neg c} \}$ by assumption.
    If $a$ is disjoint from all such $v$,
    then 
    \[
     a \wedge \neg c = a \wedge \bigvee \{v \in \O(M) \mid v \pprec \neg c\} = 0,
    \]
    so $a = 0$, a contradiction.
    Thus, there is $v \pprec \neg c$ such that $a \wedge v$ is nonzero and $v$ is disjoint from $c$.
    From $v \pprec \neg c$, by \cref{dyadic rationals}, it follows that there is a family $\{u_p\} \subseteq \O(M)$ indexed by dyadic rationals such that 
    \[
    v \leq u_0 \leq u_p \prec u_q \leq u_1 \leq \neg c \quad \text{for all } p < q.
    \]
Consequently, $a \wedge u_0$ is nonzero and $c$ is disjoint from $u_1$, and hence $M$ satisfies \ref{NT31/2}.

    (3) This is immediate since the two definitions coincide.
\end{proof}

We also have that each $T_2$-algebra satisfies \ref{NT2}:

\begin{proposition} \label{prop: T2 implies NT2}
    If $M$ is a $T_2$-algebra, then $M$ satisfies \ref{NT2}.
\end{proposition}

\begin{proof}
Let $a, b \in M$ be nonzero and disjoint.
Since $M$ is a $T_2$-algebra, there is a nonzero $T_2$-element $c$ underneath $a$. 
If $b \leq \diamond u $ for all $u \in \O(M) $ with $c \leq u$, then $b \leq \bigwedge\{ \diamond u \mid c \leq u \in \O(M) \} = c \le a$, 
a contradiction. Therefore, there is $u \in \O(M)$ such that $b \nleq \diamond u$ and $c \leq u$. 
Set $v= \neg \diamond u$. Then $v\in\O(M)$, $v$ is disjoint from $u$, and both $u \wedge a$ and $b \wedge v$ are nonzero. Thus, $M$ satisfies \ref{NT2}.
\end{proof}

On the other hand, the converse of \cref{prop: T2 implies NT2} fails. 
Indeed, the following example shows that an MT-algebra may satisfy \ref{NT2} without even being a $T_1$-algebra.

\begin{example}
    Let $B_0$ be a nontrivial complete atomless boolean algebra. Set $B = B_0 \times B_0$ and define $\square : B \to B$ by $\square(a,b) = (a \wedge b, b)$. Observe that $M := (B, \square)$ is an MT-algebra:
    \begin{itemize}
        \item $\square(1,1)=(1,1)$;
        \item $\square(a,b)=(a \wedge b, b) \leq (a,b)$;
        \item $\square \square(a,b) = \square (a \wedge b, b)=(a \wedge b, b) = \square(a,b)$;
        \item 
        $\begin{aligned}[t]
            \square (a_1,&a_2) \wedge \square (b_1,b_2) 
            = (a_1 \wedge a_2, a_2) \wedge (b_1  \wedge b_2, b_2) \\
            &= (a_1 \wedge a_2 \wedge b_1  \wedge b_2, a_2 \wedge b_2) 
            = (a_1 \wedge b_1 \wedge a_2 \wedge b_2, a_2 \wedge b_2) \\
            &= \square (a_1 \wedge b_1,a_2 \wedge b_2) 
            = \square \big((a_1,a_2) \wedge (b_1,b_2)\big).
        \end{aligned}$
    \end{itemize}
    
    Moreover, the open and closed elements of $M$ are easy to describe:
    $(a,b) \in \O(M)$ iff $a \leq b$, and hence $(a,b) \in \C(M)$ iff $(\neg a, \neg b) \in \O(M)$ iff $b \leq a$. 

    To see that $M$ is not a $T_1$-algebra, note that $(1,0)$ has no proper open elements above it since $(1,0) \leq (a,b)$ implies $a = 1$, and if $(a,b)$ is open then $a \leq b$, so $(a,b) = (1,1)$.

    To see that $M$ satisfies \ref{NT2}, let $m = (a_1, a_2)$ and $n = (b_1, b_2)$ be nonzero and disjoint.
     There are four cases possible, and in each 
     we will find disjoint $u, v \in \O(M)$ such that both $m \wedge u$ and $n \wedge v$ are nonzero.
    \begin{itemize}
        \item If $m,n \in \O(M)$, then we can set $u = m$ and $v = n$.
        \item If $m,n \not \in \O(M)$, then $a_1 \nleq a_2$ and $b_1 \nleq b_2$, so both $a_1$ and $b_1$ are nonzero. Moreover, $a_1$ and $b_1$ are disjoint since $m$ and $n$ are. 
        Thus, we can set $u = (a_1, a_1)$ and $v = (b_1, b_1)$.
        \item If $m \in \O(M)$ and $n \notin \O(M)$, then $a_1 \leq a_2$ but $b_1 \nleq b_2$. We have the following subcases:
        \begin{itemize}
            \item If $a_2$ and $b_1$ are disjoint, then we can set $u = (a_2, a_2)$ and $v = (b_1, b_1)$.
            \item If $a_2$ and $b_1$ are not disjoint, then (since $B_0$ is atomless) there is some nonzero $c \in B_0$ such that $c < a_2 \wedge b_1$. Let $d= \neg c \wedge a_2 \wedge b_1$. Then $d$ is nonzero and disjoint from $c$. Thus, we can set $u = (c, c)$ and $v = (d,d)$.
        \end{itemize}
        \item If $n \in \O(M)$ and $m \notin \O(M)$, then we proceed similarly to the previous case.
    \end{itemize}
\end{example}

   As we have just seen, \ref{NT2} is too weak to imply even the $T_1$-separation on MT-algebras. 
   It does, however, impose additional structure on the frame of opens. We recall
   (see, e.g., \cite[p.~330]{PP12}) that a frame $L$ is \emph{Hausdorff} provided every $a \in L \setminus \{1\}$ can be written as
\[
a = \bigvee \{ u \in L \mid u \leq a \text{ and } u^* \nleq a \}.\footnote{As is customary, $u^*$ denotes the pseudocomplement $u^* = \bigvee \{v \in \O(M) \mid u \wedge v = 0\}$.}
\]

\begin{theorem} \label{NT2 vs H}
    Let $M$ be an MT-algebra.
    \begin{enumerate} [ref=\thetheorem(\arabic*)]
        \item If $M$ satisfies \ref{NT2}, then $\O(M)$ is Hausdorff. \label[theorem]{NT2 implies Hausdorff frame}
        \item If $M$ is a $T_1$-algebra and $\O(M)$ is Hausdorff, then $M$ satisfies \ref{NT2}.
    \end{enumerate}
\end{theorem}

\begin{proof}
    (1) Let $1 \neq a \in \O(M)$. Suppose $a \neq \bigvee \{b \in \O(M) \mid b \leq a, b^* \nleq a\}$. Then $t := \bigvee \{b \in \O(M) \mid b \leq a, b^* \nleq a\} < a$. Therefore, $a \wedge \neg t$ and $\neg a$ are nonzero. Since $a \wedge \neg t$ and $\neg a$ are disjoint, \ref{NT2} yields disjoint $u, v \in \O(M)$ such that both
    $u \wedge (a \wedge \neg t)$ and $v \wedge \neg a$ are nonzero. Since $u,v$ are disjoint, $v \leq \square \neg u $, which gives that \[
    v \leq \square (\neg u \vee \neg a) = \square\neg(u \wedge a) = (u \wedge a)^*
    \] 
    (because the pseudocomplement in $\O(M)$ is calculated by $s^* = \square\neg s$; see, e.g., \cite[p.~125]{RS63}). Observe that $(u \wedge a)^* \nleq a$ since $v \nleq a$. Thus, $u \wedge a \leq t$, which implies that $u \wedge a \wedge \neg t=0$, a contradiction.

    (2) Let $a, b \in M$ be nonzero and disjoint. Since $M$ is a $T_1$-algebra, we may assume that $a$ and $b$ are 
    closed. Then $\neg b$ is open, and 
    from $a \wedge b = 0$ and $\O(M)$ being Hausdorff it follows that
    \[
        a \leq \neg b = \bigvee \{u \in \O(M) \mid u \leq \neg b \mbox{ and } u^* \nleq \neg b\}.
    \] 
    Therefore, there is $u \in \O(M)$ such that $u \wedge a$ is nonzero and $u^* \nleq \neg b$. The latter gives that $u^* \wedge b$ is nonzero.
    Thus, $M$ satisfies \ref{NT2}.
\end{proof}

As an immediate consequence of \cref{NT2 vs H}, we obtain:

\begin{corollary}
    Let $M$ be a $T_1$-algebra. Then $M$ satisfies \ref{NT2} iff $\O(M)$ is a Hausdorff frame.
\end{corollary}

\begin{remark} \label{rem: NT2 and T2}
The above corollary suggests that a weaker notion of Hausdorffness for an MT-algebra would be to satisfy \ref{NT1} and \ref{NT2}. In fact, this is how Nöbeling defined Hausdorffness pointfree \cite[p.~80]{Noeb54}. While this notion is useful because it captures Hausdorffness of the frame of opens, it is strictly weaker than being a $T_2$-algebra. However,  finding  
a counterexample is more involved (and uses essentially the fact, communicated to us by Alan Dow, that each compact Hausdorff extremally disconnected space is realized as the remainder of the Stone--\v{C}ech compactification of some completely regular  extremally disconnected space). The details will be presented elsewhere.
\end{remark}

\section{Local compactness, spatiality, and AC} \label{sec:4}

By the well-known Isbell Spatiality Theorem \cite[Thm.~2.1]{Isb72}, each compact subfit frame is spatial. We show
that this theorem follows from an earlier result of Nöbeling  \cite[12.5]{Noeb54}, which in the setting of MT-algebras states that each compact $T_1$-algebra is  spatial. 
We then generalize the latter by proving that already locally compact $T_\D$-algebras are spatial. This we do by showing that every nontrivial compact $T_\D$-algebra contains a closed atom, which we prove is equivalent to the axiom of choice (AC for short). Whether any of the spatiality theorems of this section are equivalent to AC remains open.  

We start by recalling that there are various definitions of local compactness in topology, which are all equivalent for Hausdorff spaces. We will mainly work with the following two:
\begin{enumerate}[\textup{(LC\arabic*)}]
    \item Each point has a compact neighborhood. \label{NLC}
    \item For each point $x$ and each open neighborhood $U$ of $x$, there is a compact neighborhood of $x$ contained in $U$. \label{LC}
\end{enumerate}

Clearly \ref{LC} implies \ref{NLC}. The converse is not true in general, but it is true for Hausdorff spaces (see, e.g., \cite[p.~130]{Wil70}). Following the standard usage in non-Hausdorff topology (see, e.g., \cite[p.~44]{GH+03} or \cite[p.~92]{GL13}), we call a space {\em locally compact} if it satisfies \ref{LC}. 

The MT-version of \ref{NLC} is defined in \cite[p.~105]{Noeb54}, and that of \ref{LC} in \cite[Def.~4.1]{BR25}.
We recall that an element $k$ of an MT-algebra $M$ is  \emph{compact} if for each family $U \subseteq \O(M)$ with $k \leq \bigvee U$, there is a finite subfamily $V \subseteq U$ such that $k \leq \bigvee V$. Then $M$ is \emph{compact} if its top element is compact. This agrees with the usual notion of compactness for frames, where a frame $L$ is \emph{compact} if $1 = \bigvee U$ implies $1 = \bigvee V$ for some finite $V \subseteq U$ (see, e.g., \cite[p.~125]{PP12}). Thus, $M$ is compact iff $\O(M)$ is compact.\label{compact sentence}

\begin{remark}
    Nöbeling essentially defined compactness using the finite intersection property for closed elements \cite[p.~95]{Noeb54}, but a standard argument shows that in the setting of MT-algebras his definition is equivalent to the one given above (see, e.g., \cite[Lem.~3.6]{BR25}).
\end{remark}
    We point out that most classical results about compact subsets extend to MT-algebras. For example, a closed element underneath a compact element is compact (see \cite[Lem.~6.11]{BR25}). We will use these results freely in what follows, often without explicit mention.

For $a,b \in M$, set $a \lhd b$ if there is a compact $k \in M$ such that $a \leq k \leq b$. 

\begin{definition} 
We say that an MT-algebra $M$ is 
\begin{enumerate}
    \item
    \emph{N-locally compact} if for each nonzero $a$, there is a compact $k \in M$ such that $a \wedge \square k$ is nonzero;
    \item
    \emph{locally compact} if $u = \bigvee \{ v \in \O(M) \mid v \lhd u \}$ for each $u \in \O(M)$.
\end{enumerate}
\end{definition}

The following is straightforward to verify.

\begin{lemma} \label{X LC iff PX LC}
    Let $M = \P(X)$ be a spatial MT-algebra.
    \begin{enumerate}
        \item $M$ is N-locally compact iff $X$ satisfies \ref{NLC}.
        \item $M$ is locally compact iff $X$ satisfies \ref{LC}.
    \end{enumerate}
\end{lemma}
The next result reformulates N-local compactness in terms closer to local compactness. While the latter requires each open element to be approximated from below via $\lhd$, N-local compactness only requires this for the top element:

\begin{proposition} \label{prop: NLC}
    $M$ is N-locally compact iff $1 = \bigvee \{u \in \O(M) \mid u \lhd 1\}$.
\end{proposition}
\begin{proof}
    ($\Rightarrow$) Let $v = \bigvee \{u \in \O(M) \mid u \lhd 1\}$. If $v \neq 1$, then $\neg v$ is nonzero. Therefore, by N-local compactness, there is a compact $k$ such that $\square k \wedge \neg v$ is nonzero. But $\square k \leq k \leq 1$, so $\square k \lhd 1$. Thus, $\square k \leq v$ and $\square k \wedge \neg v$ is nonzero, a contradiction.

    ($\Leftarrow$) Suppose $a$ is nonzero. Then $a \leq 1 = \bigvee \{u \in \O(M) \mid u \lhd 1\}$. Therefore, 
    \[
    a = a \wedge \bigvee \{u \in \O(M) \mid u \lhd 1\} = \bigvee \{ a \wedge u \mid u \in \O(M) \mbox{ and } u \lhd 1\}.
    \]
    Thus, there is $u \in \O(M)$ such that $u \lhd 1$ and $a \wedge u$ is nonzero. This means that there is a compact $k$ such that $a \wedge \square k \geq a \wedge u \neq 0$.
\end{proof}

The above proposition makes it clear that both compactness and local compactness of an MT-algebra imply its N-local compactness: 

\begin{corollary}
    Let $M$ be an MT-algebra.
    \begin{enumerate}[ref=\thelemma(\arabic*)]
        \item If $M$ is locally compact, then $M$ is N-locally compact. \label[corollary]{LC is NLC}
        \item If $M$ is compact, then $M$ is N-locally compact. \label[corollary]{compact is NLC}
    \end{enumerate}
\end{corollary}

However, the reverse implications of the above corollary are not true in general, already for spatial MT-algebras: 

\begin{example} 
    Let $X_1$ be the one-point compactification of $\mathbb Q$
    (so $X_1$ is not Hausdorff since $\mathbb Q$ is not locally compact)
    and let $X_2$ be any infinite discrete space.
    Then $X_1$ is compact but not locally compact, while $X_2$ is locally compact but not compact. We let $X$ be the disjoint union of $X_1$ and $X_2$. 
    Then $X$ is not compact since $X_2$ is not compact, and $X$ is not locally compact since $X_1$ is not locally compact. However, $X$ satisfies \ref{NLC} since both $X_1$ and $X_2$ do (because $X_1$ is compact and $X_2$ is discrete). Thus, by \cref{X LC iff PX LC}, $\P(X)$ is an N-locally compact MT-algebra that is neither compact nor locally compact.
\end{example}
For $T_2$-algebras, 
N-local compactness does imply local compactness, 
generalizing what happens in 
topological spaces:

\begin{theorem} \label{lc implies nlc}
    If $M$ is an N-locally compact $T_2$-algebra, then $M$ is locally compact.
\end{theorem}
\begin{proof}
    Let $u \in \O(M)$. By \cref{prop: NLC}, $u = \bigvee \{u \wedge v \in \O(M) \mid v \lhd 1\}$. Since $M$ is a $T_2$-algebra, it suffices to show that for each $v \lhd 1$ and $T_2$-element $a \leq u \wedge v$ there is some $v' \in \O(M)$ with $a \leq v' \lhd u$. Because $v \lhd 1$, there is a compact $k \in M$ such that $v \leq k$. By the definition of $T_2$-element,
    \[
        a = \bigwedge \{ \diamond w \mid a \leq w \in \O(M)\} \leq u \wedge v \leq k.
    \]
    Since we have 
    a meet of closed elements underneath an open element
     inside a compact element, there exist finitely many $w_1, \dots, w_n \in \O(M)$ such that $a \leq w_i$ for each $i$ and $\diamond w_1 \wedge \dots \wedge \diamond w_n \leq u$.
     Let $v' = w_1 \wedge \dots \wedge w_n$ and $k' = \diamond w_1 \wedge \dots \wedge \diamond w_n$. Then
    $
        a \leq v' \leq k' \leq u.
    $
    Moreover, $v' \in \O(M)$ and $k'$ is compact since $k'$ is a closed element underneath a compact element. Thus, $a \leq v' \lhd u$, as required.
\end{proof}

Assuming AC, the preceding result can also be derived from the spatial case 
since Nöbeling showed that AC implies that N-locally compact $T_1$-algebras are spatial. In particular, every N-locally compact $T_2$-algebra is spatial. Using \cref{X LC iff PX LC} 
(and 
that \ref{NLC} implies \ref{LC} for Hausdorff spaces), 
it then follows that N-locally compact $T_2$-algebras are locally compact. 

We now recall Nöbeling's spatiality results and derive
the well-known Isbell Spatiality Theorem as a consequence. In the rest of the section, 
we indicate by an asterisk the results that depend on AC. 

\begin{AC}
\begin{theorem}[Nöbeling's Spatiality Theorems]
\leavevmode
    \begin{enumerate}[ref=\thetheorem(\arabic*)]
        \item {\upshape (\cite[12.5]{Noeb54})} Compact $T_1$-algebras are spatial. \label[theorem]{Compact T1 is spatial}
        \item {\upshape (\cite[12.16]{Noeb54})} N-locally compact $T_1$-algebras are spatial. \label[theorem]{LC T1 is spatial}
    \end{enumerate}
\end{theorem}
\end{AC}

\begin{AC}
\begin{corollary}[Isbell's Spatiality Theorem] \label{Isbell Spatiality}
    Compact subfit frames are spatial.
\end{corollary}
\begin{proof}
    Let $L$ be a compact subfit frame. We let $M := \F L$ be the Funayama envelope of $L$. 
    Then $M$ is a $T_1$-algebra by \cref{thm: T1 subfit}, and $M$ is compact since $\O(M) \cong L$ and $L$ is compact. Therefore, $M$ is spatial by \cref{Compact T1 is spatial}. Thus, $\O(M)$ is spatial, and hence so is $L$. 
\end{proof}
\end{AC}

The following example shows that neither \cref{Compact T1 is spatial} nor \cref{LC T1 is spatial} can be generalized to $T_\D$-algebras.
\begin{example}
    Let $L$ be a nonspatial compact frame, e.g., a nontrivial complete atomless boolean algebra adjoined with a new top.
    Then $\F L$ is a compact $T_\D$-algebra (see \cref{thm: FL Td}), but $\F L$ cannot be spatial since $L$ is not
    (in fact, $\F L$ contains a single atom). Moreover, $\F L$ is N-locally compact by \cref{compact is NLC}.
\end{example}

The previous example shows that compact and hence N-locally compact $T_\D$-algebras need not contain many atoms. On the other hand, as we will see, AC implies 
that each compact $T_\D$-algebra must contain at least one atom. In fact, for this to be true it is enough to assume that the MT-algebra is $T_0$. 
However, to locate an atom underneath each nonzero element, we will require both local compactness and $T_\D$-separation. To see this, we need the following two lemmas.

\begin{AC}
\begin{lemma}
    If $M$ is a nontrivial compact MT-algebra, then $\C(M)\setminus\{0\}$ contains a minimal element. \label{min exists}
\end{lemma}
\begin{proof}
    Let $Z = \C(M) \setminus \{0\}$ and view $Z$ as a poset with the order dual to that inherited from $M$. Then $Z$ is nonempty since $1 \in Z$, and if $C \subseteq Z$ is a chain, then $\bigwedge C \neq 0$ since $M$ is compact. Therefore, $\bigwedge C \in Z$ for each chain $C \subseteq Z$. Thus,  Zorn's Lemma applies to $Z$, by which the desired element exists.
\end{proof}
\end{AC}

\begin{lemma}
    If $M$ is a $T_0$-algebra and $c \in \C(M)\setminus\{0\}$ is a minimal element, then $c$ is an atom of $M$. \label{min is atom}
\end{lemma}
\begin{proof}
    Suppose $0 \neq a $  is a $T_0$-element such that $a \leq c$. Then $a = \left(\bigwedge S\right) \wedge \neg v$ for some $S \subseteq \O(M)$ and $v \in \O(M)$. Therefore,  $\left(\bigwedge S\right) \wedge \neg v \le c$. Since $c$ is minimal in $\C(M)\setminus\{0\}$, $\neg v \wedge c \in\C(M)$, and $\neg v\wedge c$ is nonzero, we must have $\neg v \wedge c =c $. Thus, we may assume that $\neg v = c$.
    If $\bigwedge S < c$ then $c \wedge \neg\left(\bigwedge S\right)$ is nonzero, so $c \wedge \bigvee \{ \neg s \mid s \in S \}$ is nonzero. Therefore, $c \wedge \neg s$ is nonzero for some $s \in S$. But if $c \wedge \neg s$ is nonzero, then $c \wedge \neg s = c$ since $c \wedge \neg s \in \C(M)$ and $c$ is minimal. Thus, $c \leq \neg s$, so $c$ and $s$ are disjoint, and hence $a = 0$, a contradiction. Consequently, $\bigwedge S = c$. Hence, $a=c$, yielding that $c$ is an atom because $M$ is a $T_0$-algebra.
\end{proof}

Putting \Cref{min exists,min is atom} together yields:

\begin{AC}
\begin{proposition} \label{compact has atoms}
    Every nontrivial compact $T_0$-algebra contains a closed atom.
\end{proposition}

Let $M$ be an MT-algebra and $a\in M$. Then the {\em relativization} 
\[
M_a := \{b \in M \mid b \leq a\}
\]
is also an MT-algebra (see, e.g., \cite[p.~96]{RS63}), where 
\[\O(M_a) = \{u \wedge a \mid u \in \O(M)\}\quad \text{and} \quad \C(M_a) = \{c \wedge a \mid c \in \C(M)\}.\] 
Moreover, if $M$ is a $T_\D$-algebra, then so is $M_a$. To see this, every element in $M_a$ is of the form $b \wedge a$ for some $b \in M$. Therefore, if $M$ is a $T_\D$-algebra, each such $b$ is a join of $T_\D$-elements in $M$, yielding that $b \wedge a$ is a join of $T_\D$-elements in $M_a$. Thus, $M_a$ is a $T_\D$-algebra.

\begin{theorem} \label{thm: td spatial}
    If $M$ is a locally compact $T_\D$-algebra, then $M$ is spatial.
\end{theorem}
\begin{proof}
    It suffices to show that every nonzero $T_\D$-element has an atom underneath. Suppose $a \in M$ is a nonzero $T_\D$-element, so $a = u \wedge \neg v$ for some $u,v \in \O(M)$. Since $M$ is locally compact, $u = \bigvee \{w \in \O(M) \mid w \lhd u\}$. Because $a$ is nonzero, 
    there is $w \lhd u$ such that $w \wedge \neg v$ is nonzero. 
    Therefore, there is a
    compact $k \in M$ such that $w \leq k \leq u$ and $b := k \wedge \neg v$ is nonzero. Since $k$ is compact and $v$ is open, $b$ is compact. Thus, the relativization $M_b$ is a nontrivial compact $T_\D$-algebra.
    Applying
    \Cref{compact has atoms} to $M_b$ yields a closed atom $x \in M_b$. But 
    $x$ is also an atom in $M$ and $x \le a$, completing the proof. 
\end{proof}
\end{AC}

\begin{remark}
We recall 
that an MT-algebra $M$ is \emph{sober} if it is $T_0$ and for every join-irreducible element\footnote{We recall (see, e.g., \cite[p.~53]{DP02}) that a nonzero element $c$ of a lattice $L$ is {\em join-irreducible} if $c = a \vee b$ implies $c=a$ or $c=b$ for all $a,b \in L$.} $c$ in $\C(M)$ there is an atom $x\in M$ such that $c= \diamond x$.
By \cite[Thm.~6.7]{BR23}, a sober $T_\D$-algebra $M$  is spatial iff 
its frame of opens $\O(M)$ is spatial.
    Since $M$ locally compact implies that $\O(M)$ is locally compact (see \cite[Thm.~4.7]{BR25}) and hence spatial (see, e.g., \cite[Prop.~VII.6.3.3]{PP12}), we obtain that 
    every locally compact sober $T_\D$-algebra is spatial. \Cref{thm: td spatial} 
    shows that the sobriety assumption can be dropped.
\end{remark}

The spatiality of locally compact $T_\D$-algebras relies on the fact that every nontrivial compact $T_0$-algebra contains a closed atom. We conclude this section by showing that this condition is equivalent to AC.

\begin{theorem} \label{thm:AC}
    The following conditions are equivalent to \textup{AC}. 
    \begin{enumerate}[ref=\thetheorem(\arabic*)]
        \item Every nontrivial compact frame has a maximal ideal.
        \item Every nontrivial compact MT-algebra contains a nonzero minimal closed element.
        \item Every nontrivial compact $T_0$-algebra  contains a closed atom. \label[theorem]{thm:AC-3}
        \item Every nontrivial compact $T_\D$-algebra contains a closed atom.
    \end{enumerate}
\end{theorem}
\begin{proof}
    The equivalence of (1) and AC
    follows from \cite{Her03}. Thus, it suffices to show the implications (1)$\Rightarrow$(2)$\Rightarrow$(3)$\Rightarrow$(4)$\Rightarrow$(1).

    (1)$\Rightarrow$(2) Suppose $M$ is a compact MT-algebra. Then $\O(M)$ is a compact frame. By (1), $\O(M)$ has a maximal ideal $I$. Since $\O(M)$ is compact, $\bigvee I \ne 1$, so $I=\downset \bigvee I \cap \O(M)$ and $\bigvee I$ is a maximal element of $\O(M)\setminus\{1\}$ because $I$ is a maximal ideal. Thus, $c:=\neg\left(\bigvee I\right)$ is a minimal element of $\C(M)\setminus\{0\}$.   

    (2)$\Rightarrow$(3) This follows from \cref{min is atom}.

    (3)$\Rightarrow$(4) Every $T_\D$-algebra is a $T_0$-algebra.

    (4)$\Rightarrow$(1) Suppose $L$ is a nontrivial compact frame. Then the Funayama envelope $\F L$ is nontrivial. Also, $\F L$ is a $T_\D$-algebra by \cref{thm: FL Td}, and $\F L$ is compact because $\O(\F L) \cong L$ and $L$ is compact. 
    Therefore, by (4), $\F L$ has a closed atom $x$. Thus, $\neg x$ is an open coatom, and hence $I={\downarrow}(\neg x)\cap\O(\F L)$ 
    is a maximal ideal of $\O(\F L)$. 
\end{proof}

\begin{remark}
\leavevmode
\begin{enumerate}
    \item As far as we know, it remains open 
    whether any of the above spatiality results (Isbell's Spatiality Theorem, Nöbeling's Spatiality Theorems, or \cref{thm: td spatial}) is equivalent to AC. 
    \item A related question is whether the existence of maximal ideals in compact spatial frames suffices to imply AC. A positive answer would also imply that the spatial analogue of \cref{thm:AC} (in which we add the spatiality assumption to each item) 
    is equivalent to AC.
    \item
    That the existence of closed atoms in nontrivial compact $T_\D$-algebras is equivalent to AC again raises the question, posed in \cite[p.~120]{Joh82}, whether the mere existence of points in every nontrivial compact frame already implies AC. 
\end{enumerate}
\end{remark}

\section{Nonspatial locally compact sober MT-algebras}

In this final section, we show that locally compact sober MT-algebras need not be spatial, thus 
confirming the expectation of \cite[Rem.~4.14]{BR25} that the spatiality of locally compact frames does not generalize to the setting of locally compact sober MT-algebras. It also yields a negative solution to the first two open problems from  \cite[Sec.~9]{BR23}. Our construction uses the formalism of Raney extensions
\cite{Sua24,Sua25}. We work with a slight strengthening of the original definition, under which all results from \cite{Sua24,Sua25} still apply.
Let $C$ be a complete lattice. We recall that $C$ is a \emph{coframe} if finite joins distribute over arbitrary meets, and that a join $\bigvee S$ in $C$ is \emph{exact} if 
\[
    a \wedge \bigvee S = \bigvee\{ a \wedge s \mid s \in S \}
\]
for all $a \in C$. We note that exact joins are also known as {\em distributive joins} (see \cite[2.8]{BPWW16} for the history of the concept).

\begin{definition}
    A \emph{Raney extension} is a pair $R=(C,L)$ where $C$ is a coframe and $L\se C$ satisfies the following conditions:
    \begin{enumerate}
        \item $L$ is closed under arbitrary joins and finite meets;
        \item all joins in $L$ are exact in $C$;
        \item $L$ is meet-dense in $C$.
    \end{enumerate}
\end{definition}

Standard examples of Raney extensions come from topological spaces. Indeed, for each topological space $X$, the pair $\R(X) = (\S(X),\Omega(X))$ is a Raney extension, where $\S(X)$ is the coframe of saturated sets (intersections of open sets) of $X$. 
Such Raney extensions are called \emph{spatial} and are characterized as those Raney extensions $R=(C,L)$ in which the coframe $C$ is join-generated by its completely join-prime elements.\footnote{We recall (see, e.g., \cite[p.~242]{DP02}) that a nonzero element $p$ of a complete lattice $L$ is {\em completely join-prime} if for all $S \subseteq L$, from $p\le\bigvee S$ it follows that $p\le s$ for some $s\in S$.} 

There is a close connection between MT-algebras and Raney extensions. 
 For an MT-algebra $M$, we recall that $\S(M)$ is the set of saturated elements and $\O(M)$ the set of open elements, both ordered by the restriction of the order on $M$. 
 \begin{proposition}
     For an MT-algebra $M$, the pair $\R(M) = (\S(M),\O(M))$ is a Raney extension.
 \end{proposition}

\begin{proof} 
Both $\S(M)$ and $\O(M)$ are bounded sublattices of $M$. Moreover, $\S(M)$ is closed under arbitrary meets and $\O(M)$ is closed under arbitrary joins. Therefore,  
$\S(M)$ is a coframe and $\O(M)$ is a frame (in the order inherited from $M$). Moreover,  $\O(M)$ is closed under finite meets in $\S(M)$. To see that it is closed under arbitrary joins, let $S\subseteq\O(M)$. Then  
    \begin{equation*}\label{joins}
    \bigvee_{\O(M)} S = \bigwedge\{u \in \O(M) \mid s \leq u \text{ for all } s \in S \} = \bigvee_{\S(M)} S.
    \end{equation*}
    It is immediate from the definition of $\S(M)$ that $\O(M)$ is meet-dense in  $\S(M)$.  
    It is left to see that all joins in $\O(M)$ are exact in $\S(M)$. Let $a \in \S(M)$ and $S \subseteq \O(M)$. Then 
    \begin{align*}
        \bigvee_{\S(M)}\{ a \wedge s \mid s \in S\} &= \bigwedge \{u \in \O(M) \mid a \wedge s \leq u \ \ \forall s \in S \} \\
        &= \bigwedge \Big\{u \in \O(M) \mid \bigvee \{ a \wedge s \mid s \in S \} \leq u \Big\} \\
        &= \bigwedge \{u \in \O(M) \mid a \wedge \bigvee S \leq u\} \\
        &= 
        a \wedge \bigvee S
        =
        a \wedge \bigvee_{\S(M)} S,
    \end{align*} 
    where the second-to-last equality follows from $a \wedge \bigvee S \in \S(M)$. 
    Consequently, $\R(M)$ is a Raney extension.
\end{proof}

We next use the Funayama envelope to show that every Raney extension can be realized as $\R(M)$ for some $T_0$-algebra $M$, thus yielding a one-to-one correspondence between Raney extensions and $T_0$-algebras, which generalizes the one-to-one correspondence of \cref{thm: FL Td} between frames and $T_\D$-algebras.

For a Raney extension $R=(C,L)$, let $\F C$ be the Funayama envelope of $C$ (that is, $\F C$ is the MacNeille completion of the boolean envelope of $C$). Since all joins in $L$ are exact in $C$, the embedding
$L\to\F C$ has a right adjoint, which defines an interior operator $\square$ on $\F C$. 
Thus, $(\F C,\square)$ is an MT-algebra and $\O(\F C,\square) \cong L$. 

\begin{definition}
    For a Raney extension $R=(C,L)$, we call the MT-algebra $(\F C,\square)$ the {\em Funayama envelope} of $R$ and denote it by $\F R$.
\end{definition}

\begin{theorem}\label{T_0}
An MT-algebra $M$ is a $T_{0}$-algebra iff $M$ is isomorphic to $\F \R(M)$.
\end{theorem}
\begin{proof}
Let $M$ be an MT-algebra. Then the boolean envelope $\B \S(M)$ of $\S(M)$ is isomorphic to the boolean subalgebra of $M$ generated by $\S(M)$ (see, e.g., \cite[p.~99]{BD74}).
Moreover,  $\O(\F \R(M)) \cong \O(M)$ by construction. Thus, it is sufficient to show that $M$ is a $T_0$-algebra iff the boolean algebras $M$ and $\F\S(M)$ are isomorphic. For this it is enough to observe that $M$ is a $T_0$-algebra iff 
$\B \S(M)$ is join-dense in~$M$.\footnote{This follows from the well-known characterization of the MacNeille completion of a poset $P$ as the unique complete lattice $L$ in which $P$ is both join-dense and meet-dense (see, e.g., \cite[p.~237]{BD74}), and the fact that if $L$ is a boolean algebra and $P$ is a boolean subalgebra of $L$, then join-density of $P$ implies its meet-density.} 

($\Rightarrow$) Suppose $M$ is a $T_0$-algebra. Then $T_0$-elements are join-dense in $M$, and hence $\B \S(M)$ is join-dense in $M$ since $T_0$-elements sit inside $\B \S(M)$.

($\Leftarrow$) Suppose $\B \S(M)$ is join-dense in $M$. Since $\B \S(M)$ is the boolean subalgebra of $M$ generated by $\S(M)$, each element of $\B \S(M)$ can be written as 
\[
a = \bigvee_{i=1}^n (s_i \wedge \neg t_i),
\]
where $s_i, t_i \in \S(M)$ (see, e.g., \cite[p.~74]{RS63}). Therefore, it suffices to show that each element of the form $s \wedge \neg t$, with $s, t \in \S(M)$, is a join of $T_0$-elements. Since $t \in \S(M)$, it is a meet of open elements, say $t = \bigwedge u_i$. Thus,
\[
s \wedge \neg t = s \wedge \bigvee \neg u_i = \bigvee (s \wedge \neg u_i)
\]
is a join of $T_0$-elements, yielding that $M$ is $T_0$.
\end{proof}

We next show that spatiality and soberness of a $T_0$-algebra $M$ and its corresponding Raney extension $\R(M)$ go hand-in-hand. For this we recall that the {\em saturation} of $a \in M$ is given by 
\[
\s a = \bigwedge (\upset a \cap \O(M)).
\]

\begin{lemma} 
Let $M$ be a $T_0$-algebra. 
\begin{enumerate}[ref=\thelemma(\arabic*)]
    \item 
    An element $x\in M$ is an atom iff $ \forall a\in \O(M)\,(x\leq a \Longleftrightarrow x\nleq \neg a)$. \label{l:atomchar}
    \item For every completely join-prime element  $p\in\Sat(M)$ there is an atom $x\in M$ with $p=\s x$. \label[lemma]{l: from cjp to atom}
\end{enumerate}
\end{lemma}
 \begin{proof}
     (1) See \cite[Lem.~5.3]{BR+25}.
     
     (2) Let $p\in \Sat(M)$ be completely join-prime.
     We consider 
     \[
     x=p\wedge \neg\bigvee \{u\in \O(M) \mid p\nleq u\}.
     \]
     Because $p$ is completely join-prime, we cannot have $p\leq \bve\{u\in \O(M) \mid p\nleq u\}$. Thus, $x\neq 0$. We use (1) to show that $x$ is an atom. Let $a\in  \O(M)$. First, since $x \neq 0$, from $x \leq a$ it follows that $x \nleq \neg a$.
     Next, suppose that $x\nleq \neg a$. Then $\neg\bigvee \{ u\in  \O(M) \mid p\nleq u\}\nleq \neg a$, yielding that $a \not\le \bigvee \{ u\in  \O(M) \mid p\nleq u\}$. Therefore, we must have $p\leq a$, and so $x\leq a$. 
     
     It is left to show that $p=\s x$. 
     Since $p$ is saturated and $x\leq p$, we have $\s x\leq p$. 
     Suppose $x\leq a$ for some $a\in  \O(M)$. Then
     $x\nleq \neg a$, so the above argument yields $p\leq a$. This proves that $p \le \s x$, hence the equality. 
 \end{proof}
We recall (see \cite[p.~45]{Sua24}) that a Raney extension $R=(C,L)$ is \emph{sober} if for every completely prime filter $P$ of $L$ there is a completely join-prime element $p$ of $C$ such that $P=\up p \cap L$.  

\begin{theorem} \label{M spatial iff RM spatial}\label{M sober iff RM sober}
    Let $M$ be a $T_0$-algebra.
    \begin{enumerate}
        \item $M$ is spatial iff $\R(M)$ is spatial.
        \item $M$ is sober iff $\R(M)$ is sober.
    \end{enumerate}
\end{theorem}
\begin{proof}
(1) If $M$ is spatial, then $M$ is the powerset algebra $\P(X)$ for some topological space $X$. Therefore, $\Sat(M)$ is the coframe of saturated sets of $X$ and $\O(M)$ is the frame of opens of $X$. Thus, ${\sf R}(M)$ is spatial.\footnote{Observe that this implication is true for an arbitrary MT-algebra.} 
    Conversely, suppose that ${\sf R}(M)$ is spatial. Let $a \in M$ be nonzero. Since $M$ is $T_0$, there are saturated $s \in M$ and closed $c \in M$ such that $0 \ne s \wedge c \le a$. Set $u = \lnot c$. Then $s \not\le u$. Since ${\sf R}(M)$ is spatial, there is a completely join-prime $p \in \Sat(M)$ such that $p\le s$ but $p \not\le u$. By \cref{l: from cjp to atom},
    there is an atom $x$ of $M$ such that $p = \s x$. Clearly $x \le s$. If $x \not\le c$, then $x \le u$, so $p = \s x \le u$, a contradiction. Thus, $x\le s \wedge c \le a$, and hence $M$ is spatial. 
 
    (2)
    Since $M$ is a $T_0$-algebra, $M$ is sober iff every completely prime filter $P$ of $\O(M)$ is of the form $\up x \cap \O(M)$ for some atom $x\in M$ (see \cite[Lem.~5.16]{BR23}), and $\R(M)$ is sober iff $P = \up p \cap \O(M)$ for some completely join-prime $p \in \S(M)$. By \cref{l: from cjp to atom}, there is a bijection between atoms of $M$ and completely join-prime elements of $\S(M)$ given by $x\mapsto \s x$. Since $\up x \cap \O(M) = \up \s x \cap \O(M)$, the result follows.
\end{proof}

We recall (see, e.g., \cite{BPP14} or \cite{MPS20}) 
that a meet $\bigwedge S$ in $L$ is \emph{strongly exact} if from $s\to a=a$ for each $s\in S$ it follows that $(\bigwedge S)\to a=a$ for all $a\in L$.\footnote{As usual, $s\to a$ denotes the relative pseudocomplement $s\to a=\bigvee\{b\in L \mid s\wedge b \le a\}$.} 
A filter of $L$ is \emph{strongly exact} if it is closed under strongly exact meets. Let $\fse(L)$ be the set of strongly exact filters (ordered by reverse inclusion). We view $L$ as a subset of $\fse(L)$ by identifying elements of $L$ with principal filters.

\begin{proposition}
   For a frame $L$, the pair $(\fse(L),L)$ is a sober Raney extension.\label{e: SE is raney}
\end{proposition}
\begin{proof}
    By \cite[Prop.~3.10]{Sua25}, $(\fse(L), L)$ is a Raney extension.
    To see that $(\fse(L) , L)$ is sober, it suffices by 
    \cite[Prop.~3.20]{Sua24}
    to observe that every completely prime filter of $L$ is strongly exact. 
    This follows from \cite[p.~12]{Sim2004}, where the terminology of admissible filters is used. We include a short alternate proof below. 

    Let  $P\se L$ be a completely prime filter, and let $m := \bigvee \{ a \in L \mid a \notin P \}$ be the corresponding meet-prime element.\footnote{Recall (see, e.g., \cite[p.~13]{PP12}) that $m \in L \setminus \{1\}$ is {\em meet-prime} if $a \wedge b \le m$ implies $a \le m$ or $b \le m$ for all $a,b \in L$.} Then $a\in P$ iff $a\nleq m$. Since $m$ is meet-prime, we also have $a\nleq m$ iff $a\to m=m$. Let $S\subseteq P$ be a family whose meet is strongly exact. Since $s \not\le m$, we have $s\to m=m$ for each $s \in S$. Therefore, $\left(\bwe S\right)\to m=m$ because $S$ is strongly exact, so $\bwe S\nleq m$, and hence $\bwe S \in P$. Thus, $P \in \FiltSE(L)$.
\end{proof}

\begin{remark}
    Recall (see, e.g., \cite[p.~137]{PP12}) that a filter $F$ of a frame $L$ is \emph{Scott open} if for each directed set $S \subseteq L$, from $\bigvee S \in F$ it follows that $F \cap S \neq \varnothing$. 
    By \cite[Lem.~3.4(2)]{Joh85}, Scott open filters are strongly exact. Since every completely prime filter is Scott open, the above result follows. However, Johnstone's proof is rather complicated. 
    A simpler proof was given in \cite[Prop.~5.2]{JS25}, utilizing Zorn's Lemma. 
    We now indicate how to derive Johnstone's result from the above simple observation that every completely prime filter is strongly exact by only using the Prime Ideal Theorem (PIT), which is strictly weaker than Zorn's Lemma.
    For this it is sufficient to observe that PIT implies that  
    every Scott open filter is an intersection of completely prime filters (see, e.g., \cite[p.~265]{Ern18}). Therefore, since intersections of strongly exact filters are  strongly exact, we conclude that strong exactness of Scott open filters follows from that of completely prime filters.
\end{remark}

We recall that $a \in L$ is {\em dense} if $a^* = 0$, and that the set $D(L)$ of all dense elements of $L$ is a filter (see, e.g., \cite[p.~131]{RS63}).

\begin{proposition}
   For a frame $L$, $D(L)$ is a strongly exact filter.
\end{proposition}
\begin{proof}
    Let $S\subseteq D(L)$ be strongly exact. 
    Then, since $s\to 0= s^*=0$ for each $s\in S$,
    strong exactness yields $(\bwe S)\to 0 = 0$. 
    Thus, $(\bwe S)^* = 0$, and hence $\bwe S \in D(L)$. 
\end{proof}

As an immediate consequence of the above, we obtain:

\begin{corollary}\label{c: dense opens are strongly exact}
    For a topological space $X$, the set of its dense opens is a strongly exact filter of $\Om(X)$.
\end{corollary}

\begin{lemma}
Let $X$ be a topological space.
\begin{enumerate}[ref=\thelemma(\arabic*)]
    \item If $X$ is sober and singletons are nowhere dense, then $D(\Om(X))$ is not contained in any completely prime filter. \label[lemma]{l: D has no CP above it}
    \item 
    If $X$ is Hausdorff and dense-in-itself, then $D(\Om(X))$ is not contained in any completely prime filter. \label[lemma]{c: D has no CP above it}
\end{enumerate}
\end{lemma}

\begin{proof}
 (1) Since $X$ is sober, completely prime filters of the frame $\Omega(X)$ are of the form 
 \[
 N(x) := \{ U \in \Om(X) \mid x \in U \}
 \]
 for some $x\in X$ (see, e.g., \cite[Prop.~I.1.3.1]{PP12}). If $D(\Om(X))\se N(x)$
 for some $x\in X$, then $x$ is contained in all dense opens. 
 But since $\{x\}$ is nowhere dense, 
 $\overline{\{x\}}^c$ is a dense open not containing $x$, a contradiction.

 (2) Let $X$ be Hausdorff. Then $X$ is sober (see, e.g., \cite[p.~2]{PP12}). Since $X$ is dense-in-itself, all singletons in $X$ are nowhere dense. Thus, (1) applies.
\end{proof}

We are finally ready to give an example of 
a locally compact sober $T_0$-algebra which is not spatial.
\begin{example} \label{main example}
    Let $X$ be any locally compact Hausdorff dense-in-itself space. For example, we can take $X$ to be $\mathbb R$.
    The pair $R := (\fse(\Om(X)), \Om(X))$ is a sober Raney extension by \cref{e: SE is raney}. 
    To show that $R$ is not spatial, recall from \cite[Prop.~4.9]{Sua25} that $R$ is spatial iff every $F \in \FiltSE(\Om(X))$ is an intersection of completely prime filters of $\Om(X)$. 
    Consider the filter $D(\Om(X))$. By \cref{c: dense opens are strongly exact}, $D(\Om(X))$ is strongly exact; and by \cref{c: D has no CP above it}, there is no completely prime filter containing it. Thus, $R$ is not spatial.
    
    Now consider the Funayama envelope $M = \F R$. Then $\R(M) = R$, so $M$ is sober because $R$ is sober, and $M$ is nonspatial because $R$ is nonspatial (see \cref{M spatial iff RM spatial}). Moreover, $\O(M) = \Om(X)$ is locally compact since $X$ is locally compact, and hence $M$ is locally compact by \cite[Thm.~4.7(2)]{BR25}.
\end{example}

In conclusion, we return to the first two open questions posed at the end of \cite{BR23}.
The first question  asks whether a sober MT-algebra must be spatial if its frame of opens is. \cref{main example} 
provides a nonspatial locally compact sober MT-algebra $M$ whose frame of opens $\O(M)$ is spatial, thus answering the question in the negative (and also confirming the suspicion of \cite[Rem.~4.14]{BR25} that such examples exist). 

The second question asks whether a $T_0$-algebra $M$ with $\O(M)$ a subfit Hausdorff frame must be $T_2$. This is known to hold when $M$ is spatial 
(see \cite[Prop.~I.3.2]{PP21}).
\cref{main example} provides a counterexample: although $\O(M) = \Om(X)$ is Hausdorff and subfit (as it is the frame of opens of a Hausdorff space), $M$ is not even $T_\D$, let alone $T_2$. (If it were $T_\D$, it would be spatial by \cref{thm: td spatial}.) 
Moreover, the example in 
\cref{rem: NT2 and T2} indicates that
this failure persists even when $M$ is $T_1$ since in that case $\O(M)$ is Hausdorff by \cref{NT2 implies Hausdorff frame}. Thus, the spatiality assumption cannot be omitted even when restricting to $T_1$-algebras.

\addtocontents{toc}{\SkipTocEntry}
\section*{Acknowledgements}

We are very thankful to Alan Dow for communicating to us that each compact Hausdorff extremally disconnected space is realized as the remainder of the Stone--\v{C}ech compactification of some completely regular extremally disconnected space.

\bibliographystyle{alpha-init}
\bibliography{refs}

\end{document}